\documentclass[12 pt,oneside]{article}
\usepackage{amsmath}
\usepackage{amsxtra}
\usepackage{amscd}
\usepackage{amsthm}
\usepackage{amsfonts}
\usepackage{amssymb}
\usepackage{eucal}
\usepackage{hyperref}

\usepackage[margin=2.7 cm, reversemp]{geometry}

\usepackage{color}
\usepackage{diagrams}

\usepackage{graphicx}
\usepackage{setspace}

\theoremstyle{definition}
\newtheorem{thm}{Theorem}[section]
\newtheorem{cor}[thm]{Corollary}

\newtheorem{lemma}[thm]{Lemma}

\newtheorem{defn}[thm]{Definition}

\newtheorem*{rmk}{Remark}

\newtheorem*{example}{Example}

\newcommand\nc{\newcommand}
\nc\Span{\text{\rm Span}}
\nc\Id{\text{Id}}

\nc \cc {\mathbb{C}}
\nc \ff {\mathbb{F}}
\nc \nn {\mathbb{N}}
\nc \pp {\mathbb{P}}
\nc \qq {\mathbb{Q}}
\nc \rr {\mathbb{R}}
\nc \zz {\mathbb{Z}}

\nc\cA{\mathcal{A}}
\nc\cC{\mathcal{C}}
\nc\cD{\mathcal{D}}
\nc\cE{\mathcal{E}}
\nc\cF{\mathcal{F}}
\nc\cG{\mathcal{G}}
\nc\cH{\mathcal{H}}
\nc\cK{\mathcal{K}}
\nc\cM{\mathcal{M}}
\nc\cO{\mathcal{O}}
\nc\cP{\mathcal{P}}
\nc\cR{\mathcal{R}}
\nc\cS{\mathcal{S}}

\nc\fm{\mathfrak{m}}
\nc\fp{\mathfrak{p}}

\nc\Tr{\text{Tr}}
\nc\into{\hookrightarrow}

\nc\st{\text{ s.t. }}
\nc\intense[1]{\textcolor[rgb]{1.00,0.00,0.00}{\textbf{#1}}}

\nc\Mat{\text{\rm Mat}}
\nc\GL{\text{\rm GL}}
\nc\SU{\text{\rm SU}}
\nc\SO{\text{\rm SO}}
\nc\SL{\text{\rm SL}}
\nc\Sp{\text{\rm Sp}}
\nc\EL{\text{\rm EL}}

\nc\GEM{\text{\rm GEM}}
\nc\Alt{\text{\rm Alt}}
\nc\Sym{\text{\rm Sym}}

\nc\inject{\hookrightarrow}

\nc{\mattwo}[4]{\left[\begin{array}{cc} #1  & #2\\  #3 & #4 \\ \end{array} \right]}
\nc{\matthree}[9]{\left[\begin{array}{ccc} #1  & #2 & #3\\  #4 & #5 & #6 \\ #7 & #8 & #9 \\ \end{array} \right]}
\nc{\vecttwo}[2]{\left[\begin{array}{c} #1 \\ #2 \\ \end{array} \right]}
\nc{\vectthree}[3]{\left[\begin{array}{c} #1 \\ #2\\  #3 \\ \end{array} \right]}

\nc{\del}{\partial}
\nc\onto{\twoheadrightarrow}

\nc\const{\text{const}}

\nc\rrp{\rr P}
\nc\ul{\underline}
\nc\ol{\overline}
\nc\uline{\underline}
\nc\oline{\overline}
\nc\oset{\overset}
\nc\uset{\underset}

\nc\heart{\heartsuit}
\nc\spade{\spadesuit}
\nc\club{\clubsuite}

\nc\marg[1]{\marginnote{\boxed{\text{#1}}}}
\nc\margq[1]{\marginnote{\textcolor[rgb]{1.00,0.00,0.00}{#1}}}

\nc\Char{\text{Char}}
\nc\Frac{\text{Frac}}
\nc\wo{\backslash}
\nc\diag{\text{diag}}
\nc\wtl{\widetilde}
\nc\nsubgp{\triangleleft}

\nc\Cay{\text{Cay}}
\nc\Hom{\text{Hom}}
\nc\Gp{\text{Gp}}
\nc\Set{\text{Set}}
\nc\la{\langle}
\nc\ra{\rangle}

\nc\wht{\widehat}
\nc\ddx[2]{\frac{\partial {#1}}{\partial {#2}}}
\nc\dddx[3]{\frac{\partial^2 {#1}}{\partial {#2}\partial{#3}}}
\nc\mult{\text{mult}}
\nc\supp{\text{supp}}
\nc\sign{\text{sign}}

\nc\meas{\text{meas}}
\nc\tmax{\text{max}}
\nc\red{\text{red}}
\nc\dhamm{d_{\text{Hamm}}}
\nc\tr{\text{tr}}

\title{Property RD and the Classification of Traces on Reduced Group $C^*$-algebras of Hyperbolic Groups}
\author{Sherry Gong\footnote{Supported by NSF.}}
\date{}

\begin{document}
\maketitle
\onehalfspacing

\begin{abstract} In this paper, we show that if $G$ is a non-elementary word hyperbolic group, $a \in G$ is an element, and the conjugacy class of $a$ is infinite, then all traces $\tau:C^*_{\red}(G) \to \cc$ vanish on $a$. Moreover, we completely classify all traces by showing that traces $\theta:C^*_{\red}(G) \to \cc$ are linear combinations of traces $\chi_g:C^*_{\red}(G) \to \cc$ given by 
\[\chi_{g}(h)=\begin{cases} 1 & h \in C(g) \\ 0 & \text{else}\end{cases},\]
where $g$ is an element with finite conjugacy class, denoted $C(g)$. We demonstrate these two statements by introducing a new method to study traces that uses Sobolev norms and the rapid decay property.
\end{abstract}

\section{Introduction}

Let $G$ be a countable discrete group, and let $\ell^2(G)$ denote the Hilbert space of square-summable functions $G \to \cc$. Let $B(\ell^2(G))$ denominate the algebra of bounded linear operators on $\ell^2(G)$. Elements of $G$ act on $\ell^2(G)$ as bounded linear operators, so there is a natural map $\cc[G] \to B(\ell^2(G))$, where $\cc[G]$ refers to the group algebra of $G$. We define the reduced group $C^*$-algebra of $G$, written $C^*_{\red}G$, to be the operator norm closure of $\cc[G]$ in $B(\ell^2(G))$. The reduced group $C^*$-algebra plays an important role in noncommutative geometry. Connes's book, \cite{Con}, provides a panoramic account of the research surrounding such algebras.

In this paper, we delve into the study of traces on reduced $C^*$-algebras on groups. A trace on $C^*_{\red}G$ is a (not necessarily positive) bounded linear map $\tau:C^*_{\red}G \to \cc$, not necessarily positive, that satisfies $\tau(xy)=\tau(yx)$ for all $x,y \in C^*_{\red}G$. It is easy to see that if $\tau$ is such a  map, then for any $g,h \in G$, we have $\tau(ghg^{-1})=\tau(h)$. Moreover, any bounded linear map $\tau:C^*_{\red}G \to \cc$ satisfying $\tau(ghg^{-1})=\tau(h)$ is a trace.

Observe that if $\tau$ is a trace on a $C^*$ algebra, then for any complex number $\lambda$, $\lambda\tau$ is also a trace. In this article, when we say that a $C^*$ algebra has unique trace, we will mean it has unique trace up to scalar multiples.

In \cite{WY}, Weinberger and Yu studied the degree of non-rigidity of manifolds using idempotents in $C^*$-algebras of groups associated to torsion elements in those groups, and explored whether they are linearly independent in $K$-theory. One approach to the question they introduced is to use traces to show that idempotents arising from torsion elements of the group, which constitute what Weinberger and Yu called the ``finite part of $K$-theory'' in \cite{WY}, are linearly independent in $K$-theory. This approach was explored in \cite{G}. The analysis of idempotents also led to results in the study of the moduli space of positive scalar curvature metrics, as in \cite{XY1} and \cite{XY2}.

In \cite{Pow}, Powers introduced the question of exploring traces on reduced group $C^*$-algebras. Furthermore, he proved that if $G$ is a non-abelian free group, then there is only one trace on $C^*_{\red}G$. De la Harpe also investigated this subject in \cite{dlH0}, \cite{dlH1}, and \cite{dlH2}, where he showed that if $G$ is a torsion-free non-elementary hyperbolic group, then $C^*_{\red}G$ has only one trace. In \cite{AM}, Arzhantseva and Minasyan derived similar conclusions for relatively hyperbolic groups.

In this paper, we introduce a new method to examine traces using property RD. Recall that a length function on a group $G$ is a function $l:G \to \zz_{ \geq 0}$, satisfying 
\begin{itemize}
\item $l(fg) \leq l(f)+l(g)$,
\item  $l(g^{-1})=l(g)$,
\item and $l(e)=0$,
\end{itemize}
such that for $S$ a finite subset of the non-negative integers, $l^{-1}(S)$ is also finite.

Given such a length function on a group, there is a norm $\|\cdot\|_{H^s}$ on $\cc G$, namely
\[\|\sum c_g g\|_{H^s}^2 = \sum |c_g|^2(1+l(g))^{2s}.\]
Let $H^s(G)$ denote the completion of $\cc G$ with respect to this norm. Then we say $G$ has rapid decay property, or property RD, if for some $s$ there is a length function on $G$ and a constant $C$ such that	
\[\|\sum c_g g\|_{\red} \leq C\|\sum c_g g\|_{H^s}.\]

Property RD has important applications to developments pertaining to the Novikov conjecture and to the Baum-Connes conjecture, as studied by Connes and Moscovici in \cite{CM}. Jolissaint researched this property extensively in \cite{Joli1} and \cite{Joli-K}, as did Haagerup, who showed earlier in 1979, in the famous paper \cite{H}, that free groups have property RD.

For a countable discrete group $G$ with length function $l$ and a subset $X \subset G$, we say $X$ has polynomial growth if for $X_l= \{g| g \in X, l(g)=l\}$ and $n_{l}=|X_l|$, there is a polynomial $P$ satisfying $n_{l}<P(l)$. When such a $P$ does not exist, $X$ is said to grow super-polynomially. We say $X$ grows at least exponentially if there are $a,b,c$ with $a,b>0$ such that $n_{l}>ae^{bl}+c$.

In \cite{G}, we showed that if the conjugacy class of an element $g$ in a group $G$ with property RD grows super-polynomially, then any trace $\tau:C^*_{\red}G \to \cc$ vanishes on $g \in C^*_{\red}G$. We will explain this conclusion in more detail and apply it to hyperbolic groups in section \ref{on traces}, after exploring the growths of conjugacy classes in such groups in section \ref{on conjugacy classes}. In section \ref{examples}, we apply the results of section \ref{on traces} to give some examples of hyperbolic groups with unique trace and some with more than one trace.

We will use this method to classify all traces on all non-elementary hyperbolic groups, extending de la Harpe's result on torsion-free non-elementary hyperbolic groups.

The author would like to thank Prof. Guoliang Yu for his helpful advice and guidance, and Prof. David Kerr for his suggestion to look into questions regarding unique traces on reduced group $C^*$-algebras.

\section{On conjugacy classes}
\label{on conjugacy classes}	

In this section, we prove that if a conjugacy class in a hyperbolic group is infinite, then it grows exponentially. In the next section, we will use this development to find all the aforementioned traces.

Let $G$ be a finitely generated discrete group and $\delta$ be a positive real number. Consider the metric space obtained by taking the Cayley graph of $G$, and identifying each edge with the unit interval. We say $G$ is $\delta$-hyperbolic if for any $x,y,z \in G$,
\[[x,y] \subset B_{\delta}([y,z] \cup [z,x]),\]
where $[x,y]$ is the geodesic between $x$ and $y$ in the metric space, and for a subset $X \subset G$, $B_{\delta}(X)$ denotes the set of points $a$ such that there is $b \in X$ with $d(a,b)<\delta$.

We say that a group $G$ is elementary whenever $G$ contains $\zz$ as a finite index subgroup. We analyse the problem of conjugacy class growth in non-elementary hyperbolic groups.

\begin{thm} Let $G$ be a non-elementary hyperbolic group and let $a \in G$ be an element. Assume that $C(a)$, the conjugacy class of $a$ in $G$ is infinite. Then $C(a)$ grows at least exponentially.
\label{exp conj class}
\end{thm}

\begin{rmk} Note that there is a well known theorem regarding exponential conjugacy growth in hyperbolic groups, but in that theorem, ``conjugacy growth'' refers to the growth of the number of conjugacy classes that intersect a certain ball, not the number of elements in a particular conjugacy class, rendering the theorem akin to an opposite of this one.
\end{rmk}

\begin{proof}
For $t$ a non-negative real number, we designate $B_t$ the ball in the Cayley graph of $G$ around the identity of radius $t$.

To show this theorem, we want to consider some element $x \in C(a)$ of length $n$, where $n$ is large, and examine its conjugates $gxg^{-1}$, where $\ell(g)<\frac{n}{7}$. The number of $g$ with $\ell(g)<\frac{n}{7}$ grows exponentially in $n$. Here, $\ell(gxg^{-1})$ is in the range $\left[\frac{5n}{7}, \frac{9n}{7}\right]$, whereupon if all different $g$ in $\ell(g)<\frac{n}{7}$ gave rise to different $gxg^{-1}$, we would be done, by virtue of having shown that the conjugacy class of $a$ has exponentially many elements in the range.

Unfortunately, this is not quite true; there could be $g$ and $h$ with $\ell(g), \ell(h)<\frac{n}{7}$, where $gxg^{-1}=hxh^{-1}$. To surmount this problem, we will bound how much it can happen. More precisely, we will prove the following lemma:

\begin{lemma} Let $G$ be a non-elementary hyperbolic group and $a \in G$ be a fixed element. Let $x$ be in $C(a)$ and $n = \ell(x)$. Let $y \in C(a)$ satisfy $\frac{5n}{7}<\ell(y)<\frac{9n}{7}$, $m = \ell(y)$, and $k<\frac{n}{7}$ be an integer. Then the number of $h$ in $G$ such that
\begin{enumerate}
\item $\ell(h)=k$
\item $hxh^{-1}=y$
\end{enumerate}
is $O(k)$. That is, it is bounded above by $C \cdot k+C_1$, where $C$ and $C_1$ are constants that depend only on $G$ and $a$, and independent of $n$ and $k$.

Thus, the number of elements $h \in G$ such that
\begin{enumerate}
\item $\ell(h)<\frac{n}{7}$
\item $hxh^{-1}=y$
\end{enumerate}
is $O(n^2)$, meaning it is bounded by $C_2 \cdot n^2+C_3$, where $C_2$ and $C_3$ are constants that depend only on $G$ and $a$. 

\end{lemma}
\begin{proof}[Proof of Lemma] 
In triangle $ABC$ in a graph, let $a$, $b$, and $c$ denote the lengths of the sides; that is, $c=\ell(AB)$, $b=\ell(AC)$ and $a=\ell(BC)$. We say that triangle $ABC$ is $\delta_1$-thin if for points $D \in BC$, $E \in AC$ and $F \in AB$ with $\ell(AE)=\ell(AF)=\frac{b+c-a}{2}$, $\ell(BD)=\ell(BF)=\frac{a+c-b}{2}$ and $\ell(CD)=\ell(CE)=\frac{a+b-c}{2}$, we have $\ell(DE),\ell(DF),\ell(EF)<\delta_1$. Proposition 2.12 in \cite{Papa} says that for a $\delta$-hyperbolic graph, there exists $\delta_1$ such that all triangles are $\delta_1$-thin.

Let us replace $\delta$ with the max of $\delta$ and $\delta_1$, rendering both the graph $\delta$-hyperbolic and all triangles are $\delta$-thin.

Moreover, let us assume $n>5 \delta$, which we may, for there are finitely many cases for $n \leq 5 \delta$, and we can thus shove those cases into the constant $C_1$.

\begin{figure}[ht!]
\centering
\includegraphics[width=150mm]{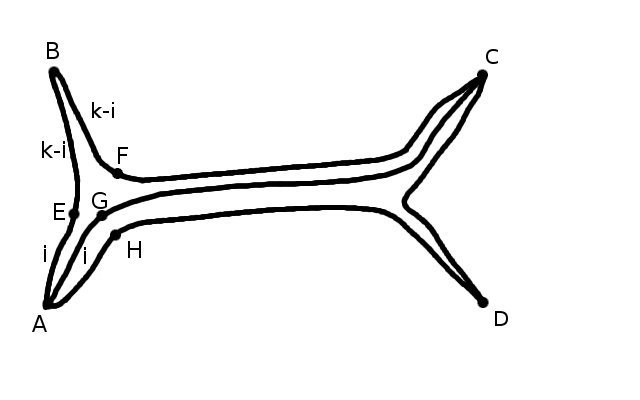}
\caption{}
\label{overflow}
\end{figure}

Choose and fix geodesics $\gamma$ from $1$ to $x$, $\rho$ from $1$ to $h$, and $\nu$ from $1$ to $y$. Let $\rho^{-1}$ denote the geodesic obtained by traversing $\rho$ backwards. Note that geodesics may not be unique, especially when there are torsion elements, but here we have fixed them ahead of time.

Consider the geodesic picture shown in the diagram, where $A=e \in G$, $B=h^{-1}$, $C=xh^{-1}$ and $D=hxh^{-1}$, and $AB$, $BC$, $CD$, and $AD$ are the geodesics $\rho^{-1}$, $\gamma$, $\rho$, and $\nu$ respectively. (In the picture, the Cayley graph is formed by taking generators $S$ and joining $g$ to $sg$ for $s \in S$. The picture is therefore invariant under the right action of $G$, so, for example, when we say that $BC$ is $\gamma$, we mean that $BC$ is the right translate of $\gamma$ by $h^{-1}$.)

Now look at triangle $ABC$. Let $i = \frac{\ell (AB)+\ell (AC)-\ell(BC)}{2}$. There are $E \in AB$, $F \in BC$ and $G \in AC$ such that $\ell(AE)=\ell(AG)= i = \frac{\ell (AB)+\ell (AC)-\ell(BC)}{2}$, $\ell(BE)=\ell(BF) = k-i = \frac{\ell (AB)+\ell (BC)-\ell(AC)}{2}$, and $\ell(CG) = \ell (CF) = \frac{\ell (AC)+\ell (BC)-\ell(AB)}{2}$. Applying that triangle $ABC$ is $\delta$-thin, we see that $\ell(EF),\ell(EG),\ell(FG)<\delta$. Applying hyperbolicity condition triangle $ACD$, we obtain that for $G \in AC$, there is $H \in AD \cup CD$ such that $\ell(GH)<\delta$. 

We claim that $H \in AD$. If not, $H$ would be in $CD$, but then $\ell(AC) \leq \ell(AG)+\ell(GH)+\ell(HC)$, $\ell(AG)=\ell(AE)<k$, and $\ell(HC)<k$, whereupon $\ell(AC)<2k+\delta$. However, $\ell(AC)>n-k$, from which we deduce $n-k<2k+\delta$, which implies $n<3k+\delta<\frac{3n}{7}+\delta$. The latter is impossible because $n>5\delta$.

Thus $H \in AD$. Ergo, we have $H \in AD$ with $\ell(EH)<\ell(EG)+\ell(GH)<2\delta$. Also, $\ell(AE)=\ell(AG)$, and consequently, 
\[\ell(AE)-\delta<\ell(AE)-\ell(GH)=\ell(AG)-\ell(GH)<\]
\[\ell(AH)<\ell(AG)+\ell(GH)=\ell(AE)+\ell(GH)<\ell(AE)+\delta.\]

Hence, for fixed $x$ and $y$, and for fixed $i$ with $0 \leq i \leq k$, which is the length of $AE$, we will estimate an upper bound for the number of possible configurations as in the diagram. We have fixed $F \in BC$, by way of the condition $\ell(BF)=\ell(BE)=k-\ell(AE)=k-i$, and $BC$ has to be the $h^{-1}$ right translate of the geodesic for $x$. In consequence, $BF$ is fixed; it is just the product of the first $k-i$ edges of the geodesic of $x$.

$BE$, as a group element, has distance at most $\delta$ from $BF$, meaning that there are at most $|B_\delta|$ possibilities for $BE$.

For $i$ fixed, as above, there are at most $2 \delta$ possibilities for $AH$, due to the fact that $H$ is along $AD$, which is the geodesic of $y$, and we showed above that $i-\delta=\ell(AE)-\delta<\ell(AH)<\ell(AE)+\delta=i+\delta$. Accordingly, $AH$ is the product of the first $t$ edges of the geodesic of $y$.

There are at most $2\delta$ choices of $AH$, and for each such choice there are at most $|B_{2\delta}|$ choices of $AE$, for $AE$ and $AH$ are within $2\delta$ of each other as group elements. Thus, in total, for fixed $i$, there are at most $2\delta |B_{2\delta}|$ possibilities for $AE$.

Consequently, for fixed $i = \ell(AE)=k-\ell(BE)$, there are at most $2\delta |B_{2\delta}|$ choices of $AE$ and $|B_{\delta}|$ choices of $BE$, which implies that there are at most $2\delta |B_{\delta}| |B_{2\delta}|$ choices of $AB=EB \cdot AE$. But $AB=h^{-1}$, meaning there are at most $2\delta|B_{\delta}||B_{2\delta}|=O(1)$ choices of $h$ for a given $k$ and $i$.

There are $k$ choices of $i \in [0,k]$, and accordingly there are in total at most $k(2\delta |B_{\delta}| |B_{2\delta}|)$ possible $h$ for a given $k$, allowing us to deduce the first statement in the lemma.

Summing up these values for different $k<n/7$ is adding up $n/7$ summands, each of which is $O(n)$, resulting in $O(n^2)$ possible $h$ with $\ell(h)<n/7$ and $hxh^{-1}=y$.

\end{proof}

From here, we complete the proof of the theorem. Observe that for any $n>\ell(a)$, there is an element $x \in C(a)$ with $\ell(x) =n $ or $\ell(x) = n+1$. This existence appears as a consequence of the presence of $y \in C(a)$ with $\ell(y)>n+1$, and our ability to write $y = s_{N} s_{N-1} \cdots s_1 a s_1^{-1} \cdots s_N^{-1}$, where each $s_i$ is one of the generators. Then
\[|\ell(s_{k+1} \cdots s_1as_1^{-1} \cdots s_{k+1}^{-1}) - \ell(s_{k} \cdots s_1as_1^{-1} \cdots s_{k}^{-1})| \leq 2\]
and $\ell(a)<n$, from which we deduce that there is $k$ such that $\ell(s_{k} \cdots s_1as_1^{-1} \cdots s_{k}^{-1})  \in \{n,n+1\}$.

Let $n>\ell(a)+1$, satisfy that there is some $x \in C(a)$ with $\ell(x)=n$. Then for $g \in B_{n/7}$, we may consider $gxg^{-1} \in C(a)$. Then $5n/7<\ell(gxg^{-1})<9n/7$, as a result of which we can map $B_{n/7}$ to the elements of the conjugacy class of $a$ of length in the range $\left[\frac{5n}{7}, \frac{9n}{7}\right]$:
\[\phi: B_{n/7} \to C(a) \cap (B_{9n/7} - B_{5n/7}).\]

By the above lemma, the pre-image of any element under $\phi$ has size at most $C_2\cdot n^2+C_3$, for constants $C_2,C_3$ depending only on $a$ and $G$. However, $|B_{n/7}|$ grows exponentially by Koubi's theorem (Theorem 1.1 in \cite{K}), which states that non-elementary hyperbolic groups have exponential growth. As a result, the number of elements in the image of $\phi$ is bounded below by $C\cdot e^{\alpha n}/n^2+D$ for constants $C$ and $D$.

Since for any $n$, there is $x \in C(a)$ with length either $n$ or $n+1$, we get that the conjugacy class has exponential growth, as desired.

\end{proof}

\begin{cor} For $G$ a non-elementary hyperbolic group and $a \in G$ a non-torsion element, the conjugacy class of $a$ grows exponentially.
\label{cor non tors conj exp}
\end{cor}
\begin{proof} By the theorem, it suffices to show that the conjugacy class is infinite. Let $C(a)$ be the conjugacy class of $a$ and $Z(a)$ be the centraliser. The fact that $C(a) = G/Z(a)$ implies that if the conjugacy class is finite, then $Z(a)\subset G$ has finite index. By Theorem 3.35 of \cite{Papa}, however, the centraliser of an infinite order element in a hyperbolic group is a finite extension of the subgroup generated by the element. As a result, $Z(a)$ is a finite extension of $\la a \ra$, and is thus virtually cyclic. Applying that any group is quasi-isometric to any of its finite index subgroups, we would then have derived that $G$ is also virtually cyclic, resulting in the desired contradiction.
\end{proof}

\section{On traces}
\label{on traces}

In this section, we show the main result; we classify all traces on non-elementary word-hyperbolic groups.

\begin{thm} For $G$ a non-elementary word-hyperbolic group, the traces $\theta:C^*_{\red}(G) \to \cc$ are exactly the linear combinations of $\chi_g:C^*_{\red}(G) \to \cc$ given by 
\[\chi_{g}=\begin{cases} 1 & g \in C(g) \\ 0 & \text{else}\end{cases},\]
where $g$ is an element of finite conjugacy class and $C(g)$ is its conjugacy class.
\label{all traces}
\end{thm}
\begin{proof}[Proof of Theorem \ref{all traces}] As was shown in \cite{Papa} [Theorem 3.27], in a hyperbolic group, there are only finitely many conjugacy classes of torsion elements. We demonstrated in the proof of Corollary \ref{cor non tors conj exp} that conjugacy classes of non-torsion elements are infinite. As consequence, there are only finitely many elements that have finite conjugacy class.

For each such element $g$, it is easy to see that the map
\[\chi_{g}=\begin{cases} 1 & g \in C(g) \\ 0 & \text{else}\end{cases}\]
defines a bounded linear functional $C^*_{\red}(G) \to \cc$, and is consequently a trace.

To verify that every bounded trace is a linear combination of the finitely many aforementioned traces, we will use the fact that word hyperbolic groups have property RD:

\begin{defn} Given a length function on a group, there is a norm $\|\cdot\|_{H^s}$ on $\cc G$ given by
\[\|\sum c_g g\|_{H^s}^2 = \sum |c_g|^2(1+l(g))^{2s}.\]
Let $H^s(G)$ denote the completion of $\cc G$ with respect to the aforementioned norm. Then we say $G$ has rapid decay property, or property RD if for some $s$, there is a length function on it and a constant $C$ such that
\[\|\sum c_g g\|_{\red} \leq C\|\sum c_g g\|_{H^s}.\]
\end{defn}

It was shown in \cite{Joli1} that all hyperbolic groups have property RD.

We will now apply Lemma 2.8 in \cite{G}, which states the following:

\begin{thm}
Let $G$ be a group with property RD and $g \in G$ be an element whose conjugacy class grows faster than polynomially. Then all traces $C^*_{\red}G \to \cc$ vanish on $g$.
\end{thm}

In our case, word-hyperbolic groups have property RD. Ergo, applying Theorem \ref{exp conj class}, we have that all traces $C^*_{\red}G \to \cc$ vanish on all elements with infinite conjugacy class, and they must be linear combinations of the above traces, as desired.

\end{proof}

In particular, this gives another proof of de la Harpe's result regarding traces on torsion-free non-elementary word-hyperbolic groups, which can be found in \cite{dlH1}:

\begin{cor} For $G$ a torsion-free, non-elementary, word-hyperbolic group, there is only one trace $\tau:C^*_{\red}G \to \cc$.
\end{cor}

Theorem \ref{all traces} also implies that every hyperbolic group can be written as an extension of a hyperbolic group with unique trace by a finite group. More precisely:

\begin{thm} Let $G$ be a hyperbolic group. Then it has a finite normal subgroup $N$ such that $G/N$ is a hyperbolic group with unique trace.
\label{normal subgp}
\end{thm}

\begin{proof} Let $N \subset G$ be the subset consisting of elements with finite conjugacy class. Note that all elements in it have finite order, as we showed in the proof of Theorem \ref{all traces}. 

\begin{lemma}
We claim that $N$ is a subgroup.
\end{lemma}
\begin{proof} It is easy to see that $N$ is closed under inversion, and it therefore suffices to show that if $g_1$ and $g_2$ are in $N$, then $g_1g_2$ is in $N$. For an element $g \in G$, let $Z(g) \subset G$ denote the subgroup consisting of elements that commute with $g$. Observe that $g$ is in $N$ if and only if $Z(g)$ has finite index as a subgroup of $G$. Thus, it suffices to show that if $Z(g_1)$ and $Z(g_2)$ have finite index, then $Z(g_1g_2)$ does as well.

Note that $Z(g_1g_2)$ contains the intersection of $Z(g_1)$ and $Z(g_2)$, rendering it sufficient to show that if $H_1$ and $H_2$ are finite index subgroups of $G$, then $H_1 \cap H_2$ is as well. There is a natural injection $H_1/H_1 \cap H_2 \to G/H_2$. Thus, $|H_1/H_1 \cap H_2 |<\infty$. Let $a_1,a_2,\ldots a_n$ and $b_1,b_2,\ldots b_m$ satisfy
\[G = a_1H_1 \cup a_2 H_1 \cup \cdots \cup a_n H_1\]
and
\[H_1 = b_1 (H_1 \cap H_2) \cup b_2 (H_1 \cap H_2) \cup \cdots \cup b_m (H_1 \cap H_2).\]
Then $G = \bigcup_{i,j} a_ib_j(H_1 \cap H_2)$. Therefore, $H_1 \cap H_2$ has finite index, as desired.
\end{proof}

By the lemma, we have that $N$ is a subgroup of $G$. As we determined in the proof of Theorem \ref{all traces}, $N$ is finite. Also, it is easy to see that $N$ is closed under conjugation. Hence, $N$ is a finite normal subgroup.

Note that $G/N$ is hyperbolic, on the grounds that hyperbolicity is preserved under quasi-isometries, and it is easy to see that the map $G \to G/N$ is a quasi isometry. Thus, $G/N$ is hyperbolic. It remains to show that $G/N$ has unique trace $\tau:C^*_{\red}(G/N) \to \cc$. Such a trace $\tau$ induces a trace $C^*_{\red}(G) \to \cc$, because $N$ is a finite normal subgroup. By Theorem \ref{all traces}, the trace $C^*_{\red}(G) \to \cc$ vanishes on all elements outside of $N$; ergo the map $\tau: C^*_{\red}(G/N) \to \cc$ vanishes on all elements except the identity, as desired.

\end{proof}

\section{Examples}
\label{examples}

In this section, we apply Theorem \ref{all traces} to give some examples of hyperbolic groups with torsion whose reduced $C^*$-algebras have unique trace, and some examples in which the reduced $C^*$-algebras have more than one trace.

\begin{example} Let $G = G_1 * G_2$ be the free product of two non-trivial hyperbolic groups $G_1$ and $G_2$. Suppose that $|G_1|>2$.\footnote{This assumption does not lose any generality, because $\zz/2 * \zz/2$ is elementary.} Then $G$ is hyperbolic and all its conjugacy classes are infinite. The hyperbolicity results from the general fact that the free product of any two hyperbolic groups is hyperbolic. To see that the conjugacy classes have infinite order:

Observe that every element $x \in G$ can be uniquely expressed as $x=g_1h_1g_2h_2\cdots g_nh_n$, where $g_i \in G_1$, $h_i \in G_2$ and $g_i \neq 1$ for $i>1$ and $h_i \neq 1$ for $i<n$.

In the case $g_1 = 1$ and $h_n \neq 1$, consider some element $g \neq 1 \in G_1$ and $h \neq 1 \in G_2$, and consider $hghg \cdots hghgxg^{-1}h^{-1} g^{-1}h^{-1}\cdots g^{-1}h^{-1} = (hg)^tx(hg)^{-t}$, that is conjugation of $x$ by $(hg)^t$. As $t$ varies amongst positive integers, $(hg)^tx(hg)^{-t}$ goes through pairwise distinct elements of the conjugacy class of $x$ in $G$.

The case where $g_1  \neq 1$ and $h_n = 1$ is similar.

For the case $g_1 \neq 1$, and $h_n \neq 1$, conjugate by $(hg)^t$, where $g \in G_1$ has $g \neq 1$ and $g \neq g_1^{-1}$. Then the same argument as above applies. Similarly, for the case $g_1 =1$ and $h_n=1$, pick $g \in G_1$, with $g \neq 1$ and $g \neq g_n^{-1} \in G_1$.

Consequently, we have that $G$ is a hyperbolic group with infinite conjugacy classes, implying, by Theorem \ref{all traces}, that whenever such a group is non-elementary, its reduced $C^*$-algebra has unique trace.

This example verifies that many hyperbolic groups with torsion elements still have reduced $C^*$-algebras with unique traces, extending de la Harpe's results in \cite{dlH0}, \cite{dlH1}, and \cite{dlH2}, which states that torsion-free non-elementary hyperbolic groups have reduced $C^*$-algebras with unique traces.

\end{example}

Now let us give an example of a non-elementary hyperbolic group that does not have unique trace. For this, by Theorem \ref{all traces}, it suffices that there are non-identity elements with finite conjugacy class.

There are obvious examples of such groups:
\begin{example} Let $G$ by a non-elementary hyperbolic group and $G_1$ be a non-trivial finite group. Then $G_1 \times G$, the direct product of $G_1$ and $G$, is hyperbolic, for it is quasi-isometric to $G$. The subset $G_1 \times e$ of $G_1 \times G$, where $e$ is the identity in $G$, is clearly preserved under conjugation. It is moreover finite, implying that any element in it has finite conjugacy class.
\end{example}

Here is an example of a non-elementary hyperbolic group that does not have only one trace, but which is not the direct product of a finite group and a hyperbolic group with unique trace.

\begin{example} Let $G=\la x,y,a \ra/\la a^3,aya^{-1}y^{-1}, xax^{-1}a^{-2}\ra$. This group $G$ is hyperbolic, in that it is quasi-isometric to the free group on two generators, $\la x,y,\ra \subset G$. In $G$, $a$ and $a^{2}$ clearly have finite conjugacy class, as their conjugacy class is $\{a,a^2\}$.  Thus, by Theorem \ref{all traces}, $G$ does not have unique trace. 

However, it is not the direct product of a finite group and a hyperbolic group whose reduced $C^*$-algebra has unique trace, because the only elements of finite order in $G$ are $a$ and $a^2$; ergo, if $G$ were such a direct product, it would be $\zz/3\zz \times G_1$ for some hyperbolic group $G_1$, but in that case, $a$ would not be conjugate to $a^2$.
\end{example}


\noindent Department of Mathematics, Massachusetts Institute of Technology,

\noindent Cambridge, MA 02139, USA

\noindent E-mail: \url{sgongli@mit.edu}

\end{document}